\documentclass[12pt]{amsart}

\usepackage[margin=1in,letterpaper,portrait]{geometry}
\usepackage{amsmath}
\usepackage{amssymb}
\usepackage{amsthm}
\usepackage{mathrsfs}
\usepackage{verbatim}
\usepackage{fancyhdr}
\usepackage{url}
\usepackage{tikz}
\usepackage{enumerate}

\usetikzlibrary{matrix,arrows}
\usepackage{tikz-cd}
\usetikzlibrary{plotmarks}
\usepackage{pgfplots}
\usepgfplotslibrary{fillbetween}

\usepackage{setspace}
\usepackage{cleveref}

\newcommand{\mb}{\mathbb}

\newcommand{\mc}{\mathcal}

\newcommand{\ol}{\overline}
\newcommand{\wt}{\widetilde}

\newcommand{\Hilb}{{\rm Hilb}}
\newcommand{\RHilb}{\wt{{\rm Hilb}}}

\newcommand{\Spec}{{\rm Spec}}

\newcommand{\ms}{\mathscr}
\newcommand*{\sheafhom}{\mathscr{H}\kern -.5pt om}
\newcommand{\OO}{\mathscr{O}}

\newcommand{\NE}{{\rm NE}}
\newcommand{\Nef}{{\rm Nef}}
\newcommand{\Curv}{{\rm Curv}}

\begin{document}
\theoremstyle{plain}
\newtheorem{Thm}{Theorem}[section]
\newtheorem{Cor}[Thm]{Corollary}
\newtheorem{Conj}[Thm]{Conjecture}
\newtheorem{Pro}[Thm]{Problem}
\newtheorem{Main}{Main Theorem}
\renewcommand{\theMain}{}
\newcommand{\Sol}{\text{Sol}}
\newtheorem{Lem}[Thm]{Lemma}
\newtheorem{Claim}[Thm]{Claim}
\newtheorem{Prop}[Thm]{Proposition}
\newtheorem{ToDo}{To Do}

\theoremstyle{definition}
\newtheorem{Exam}[Thm]{Example}
\newtheorem{Def}[Thm]{Definition}
\newtheorem{Exer}[Thm]{Exercise}
\newtheorem{Rem}[Thm]{Remark}

\theoremstyle{remark}

\title{Complete families of immersed curves}
\author{Dennis Tseng}
\address{Dennis Tseng, Harvard University, Cambridge, MA 02138}
\email{DennisCTseng@gmail.com}
\date{\today}

\allowdisplaybreaks

\begin{abstract}
We extend results of Chang and Ran regarding large dimensional families of immersed curves of positive genus in projective space in two directions. In one direction, we prove a sharp bound for the dimension of a complete family of smooth rational curves immersed into projective space, completing the picture in projective space. In another direction, we isolate the necessary positivity condition on the tangent bundle of projective space used to run the argument, which allows us to rule out large dimensional families of immersed curves of positive genus in generalized flag varieties.
\end{abstract}

\maketitle

\section{Introduction}
One might expect that a complete family of curves will have singular elements. A natural question is: Do there exist nontrivial complete families of smooth curves, and, if so, what is the largest possible dimension of such a family? 


Our starting point is the following sharp result by Chang and Ran
\begin{Thm}
[{\cite[Theorem 1]{CR2}}]
\label{positivegenus}
The largest dimension of a complete family of smooth genus $g$ curves immersed into $\mb{P}^n$ is $n-2$ for $g>0$. 
\end{Thm} 

Using the same methods, we will extend Theorem \ref{positivegenus} in two directions. First, we extend Theorem \ref{positivegenus} to the case of genus zero curves, and obtain a sharp result. This is new even in the case of conics, for example ruling out 2-parameter families of smooth conics in $\mb{P}^3$ as asked by DeLand \cite[Question 1.4]{deLand}. 
\begin{Thm}[Theorem \ref{PnT}]
\label{CRC}
The largest dimension of a complete family of smooth rational curves immersed into $\mb{P}^n$ is dimension $n-2$, except in the case of lines.
\end{Thm} 

In another direction, we can try to generalize the ambient variety from $\mb{P}^n$ to a smooth, $n$-dimensional variety $X$. In this case, we can show that strong conditions on the positivity of the tangent bundle of $X$ yield the same upper bound in the case of positive genus curves. 
Over $\mb{C}$, the only varieties that can satisfy this positivity condition are products of generalized flag varieties, and in fact they do.
\begin{Thm}[Theorem \ref{thmGP}]
\label{GPThmI}
Let $X=G_1/P_1\times\cdots\times G_i/P_i$ be a product of generalized flag varieties over $\mb{C}$. If $X$ is $n$-dimensional, then any complete family of smooth, positive genus curves immersed into $X$ is dimension at most $n-2$. 
\end{Thm}

As suggested by the statement, the bound in Theorem \ref{positivegenus} and Theorem \ref{CRC} is sharp. One can construct complete $n-2$ dimensional families of smooth curves immersed into $\mb{P}^n$ of any genus and any degree by embedding the total space by a very ample bundle into a large projective space, and then projecting down to $\mb{P}^n$ via a generic projection \cite[Example 2.2]{CR1}. However, we suspect that the bound in Theorem \ref{GPThmI} is not optimal in general.

\subsection{Related results}
\label{relatedr}
Since Theorem \ref{positivegenus} applies only to positive genus curves, we collect a couple of results in the literature that can be applied to rational curves. First, Chang and Ran have an earlier result that applies to curves of any genus.

\begin{Thm}
[{\cite[Theorem 2.1]{CR1}}]
\label{CR1}
A complete family of smooth curves (of any genus) embedded into $\mb{P}^n$ is dimension at most $n-1$, except for the case of lines.
\end{Thm}
Theorem \ref{CR1} is stated in a slightly weaker form for simplicity. The original statement allowed for finitely many total singular points of the image curves across \emph{all} the members of the family. 

Second, there is a result of DeLand on complete families of rational curves whose span is the maximum possible for the degree (e.g. rational curves that are rational normal curves in a linear subspace), which is proven by pulling back ample divisor classes from a flag variety \cite[Theorem 1.1]{deLand}. Even though his result is implied by Theorem \ref{CR1}, there is a discussion of interesting problems where one might expect better bounds than Theorem \ref{CRC} given his stronger assumptions \cite[Section 1.1]{deLand}.

\subsection{Acknowledgements}
The author would like to thank Ziv Ran for helpful comments and pointing out that the proof of Lemma \ref{nef} can be simplified and David Yang for making the author aware of the problem. The author would also like to thank Paolo Aluffi for helpful discussions and for the references on the positivity of the chern classes of a generalized flag variety. 

\section{General setup}
\label{setup}
 We revisit the basic setup by Chang and Ran \cite{CR2}, before applying it in two situations. However, we first need to precisely define what we mean by a complete family of curves. 

For example, in our definition of a complete family of immersed curves, we don't want to include the case of a family that is pulled back from a smaller dimensional base.


\begin{Def}
\label{familydef}
By a complete family of immersed curves in a smooth variety $X$, we mean a diagram
\begin{center}
\begin{tikzcd}
\ms{C} \arrow[r,"f"] \arrow[d,"\pi"] & X\\
B &
\end{tikzcd}
\end{center}
where $B$ is integral and proper over $k$, $\pi$ is proper and smooth of relative dimension 1 with connected fibers, and $f$ is an immersion when restricted to each fiber $C_b$ of $\pi$ for $b\in B$. 

As in \cite{CR2}, we impose the condition that the rational map $B\dashrightarrow {\rm Hilb}_{X}$ is generically finite. More precisely, there is a dense open set $U\subset B$ for which there is a map $U\to \Hilb_X$ that sends $b\in B$ to the image $f(\pi^{-1}(b))\subset X$. See Proposition \ref{HilbIm} below for a proof of the existence of such a map. 


This condition of the family not arising from pullback of a smaller dimensional base is also referred to as the family being \emph{nondegenerate} \cite{CR2} or of \emph{maximal moduli} \cite{deLand}. The dimension of the family is the dimension of $B$.  
\end{Def}

For the rest of the paper, by a complete family, we mean a complete family of curves in the sense of Definition \ref{familydef}.

\subsection{Conventions}
Unless stated otherwise, we will let
\begin{enumerate}
\item
our base field is an algebraically closed field $k$ of arbitrary characteristic
\item
$X$ is a smooth variety of dimension $n$
\item
$\pi:\ms{C}\to B$ and $f:\ms{C}\to X$ give a complete family of curves
\item
$K$ is the divisor associated to the relative sheaf of differentials of $\pi: \ms{C}\to B$
\item 
$T_{\pi}$ is the relative tangent sheaf of $\ms{C}\to B$, so in particular $K=-c_1(T_{\pi})$
\end{enumerate}

When we take the projectivization $\mb{P}(V)$ of a vector bundle, we mean the projective bundle of lines in $V$, rather than 1-dimensional quotients. For convenience and to avoid notational clutter, we will refer to $K$ and all its pullbacks as just $K$. Similarly, if $X=\mb{P}^n$ and $H$ is the hyperplane class, we will often drop the pullback notation when the choice of map is unambiguous. 
\subsection{Chern class obstruction}
First, the obstruction to the family $\ms{C}$ is the following equality, which is a straightforward generalization of the the case $X=\mb{P}^n$ in \cite[Proof of theorem]{CR2}. 
\begin{Lem} 
\label{obstruction}
Given a complete family of curves $f:\ms{C}\to X$ in a smooth variety $X$, we must have the identity
\begin{align*}
(K^n + f^{*}c_1(TX)K^{n-1}+\cdots+f^{*}c_n(TX))\cap [\ms{C}]&=0
\end{align*}
on the total space $\ms{C}$. 
\end{Lem}

\begin{proof}
The map $f:\ms{C}\to X$ induces a map of vector bundles $T_{\pi}\to f^{*}T X$ over $\ms{C}$, where $T_{\pi}$ is the relative tangent sheaf to the family $\ms{C}\to B$. The assumption that $\ms{C}\to X$ is immersed on each fiber implies that this map is nonvanishing everywhere. In other words, we must have
\begin{align*}
c_{n}(f^{*}TX\otimes T_{\pi}^{*})\cap [\ms{C}]&=0\\
(K^n + f^{*}c_1(TX)K^{n-1}+\cdots+f^{*}c_n(TX))\cap [\ms{C}]&=0.
\end{align*}
\end{proof}

\subsection{Gauss map and curvilinear schemes}
We have to somehow use the fact that a family of curves in Definition \ref{familydef} has the property that $f: \ms{C}\to X$ induces a generically finite rational map $B\dashrightarrow {\rm Hilb}_{X}$. The essential idea is that the truncation map from a space of curvilinear subschemes of $X$ has fibers that are affine spaces, and so the fibers cannot contain positive dimensional complete subvarieties. This idea was already used to prove a similar statement \cite[Proposition 1.5]{CR1}.

The goal is to show that the Gauss map is generically finite after restricting to a general point of the image, and this is given by Lemma \ref{curvilinear}. 
\begin{Lem} 
\label{curvilinear}
Given a complete family of curves $\ms{C}\to X$, for general $p\in f(\ms{C})$, the Gauss map $\gamma: f^{-1}(p)\to \mb{P}(T_pX)$ that sends each point $q\in f^{-1}(p)$ to the 1-dimensional subspace of $T_pX$ corresponding to the image of $T_qC_{\pi(q)}$ in $T_pX$ under $f$ is generically finite. 
\end{Lem}
The proof of Lemma \ref{curvilinear} is in Appendix \ref{appendix1}. The idea is unchanged, but we cannot just cite \cite[Proposition 1.5]{CR1} because our varieties are parameterized by the map $f$ instead of embedded in $X$ and the image of $f$ in $B\times X$ could both be singular and not flat over $B$. 

Lemma \ref{curvilinear} will be applied in the form of Corollary \ref{gauss}. 
\begin{Cor}
\label{gauss}
As in Lemma \ref{curvilinear}, let  $\ms{C}\to X$ be a complete family of curves,  $p\in f(\ms{C})$ be general and $\gamma: f^{-1}(p)\to \mb{P}(T_pX)$ be the Gauss map, if $b$ is the generic fiber dimension of $f:\ms{C}\to X$, then
\begin{align*}
\int_{f^{-1}(p)}{K^{b}}>0. 
\end{align*}
\end{Cor}

\begin{proof}
Consider the map $\gamma: f^{-1}(p) \to \mb{P}(T_pX)$ that sends each point $q\in f^{-1}(p)$ to the 1-dimensional subspace of $T_pX$ corresponding to the image of $T_qC_{\pi(q)}$ in $T_pX$ under $f$.  Since the map $\gamma$ is generically finite by Lemma \ref{curvilinear}, this implies 
\begin{align*}
\int_{f^{-1}(p)}(\gamma^{*}H)^{b}>0,
\end{align*} 
where $H$ is the hyperplane class in $\mb{P}(T_pX)$. To finish, we want to identify $f^{*}H$ with $K$. 

By considering the map $T_{\pi}|_{f^{-1}(p)}\to T_pX$, we see the condition of lying in a hyperplane is the condition that $T_{\pi}|_{f^{-1}(p)}\to k$ vanishes, where $k$ is a 1-dimensional vector space corresponding to a 1-dimensional quotient of $T_pX$. Furthermore, the zero locus of the map $T_{\pi}|_{f^{-1}(p)}\to k$ vanishes in the correct dimension since $\gamma$ is generically finite. Since this is the zero locus of a section of $T_{\pi}^{*}|_{f^{-1}(p)}$, the hyperplane class in $\mb{P}(T_{p}X)$ pulls back to the class of $K$ on $f^{-1}(p)$.
\end{proof}

\section{Positivity of the tangent bundle}
We identify a condition on the tangent bundle of $X$ that allows us to use the setup in Section \ref{setup} to rule out large families of positive genus curves immersed into $X$. 
\begin{Thm}
\label{posTX}
Suppose $X$ is a smooth, $n$-dimensional projective variety such that $TX$ is globally generated and has the property that
\begin{align*}
\int_{X}{c_i(X)\cap [Z]}>0
\end{align*}
for all effective cycles $[Z]$ of dimension $i$. Then, any complete family of positive genus curves in $X$ is dimension at most $n-2$. 
\end{Thm}

Since the proof of Theorem \ref{posTX} is similar to \cite[Theorem 1]{CR1}, we will not include it here, but put it in Appendex \ref{appendix2} for completeness. The assumption of positive genus is used to conclude that $K$ is numerically effective \cite[Theorem 4.1]{Mumford}.

Suppose the base field is $\mb{C}$. Then, is a smooth variety $X$ with globally generated tangent bundle is a homogenous space \cite[Proposition 2.1]{Campana} and isomorphic to the product of an abelian variety and a product of generalized flag varieties $G_i/P_i$ by the Borel-Remmert theorem \cite{BR}. 

An abelian variety has trivial tangent bundle, so it fails the additional positivity imposed by Theorem \ref{posTX}. Indeed, the conclusion of Theorem \ref{posTX} fails for an abelian variety as translating within an abelian variety of dimension $n$ can yield $n$-dimensional families of embedded curves. Therefore, it remains to ask whether a product of generalized generalized flag varieties over $\mb{C}$ satisfies the conditions of Theorem \ref{posTX}. We will show that it does in Section \ref{GP}. 


\begin{Exam} (Quadric hypersurfaces)
Here is a simple proof communicated by Paolo Aluffi that a smooth quadric hypersurface $Q\subset\mb{P}^{n+1}$ satisfies the conditions of Theorem \ref{posTX}. Over $\mb{C}$, this is a special case of Section \ref{GP}. 

From the short exact sequence $0\to TQ\to T\mb{P}^{n+1}|_{Q}\to N_{Q/\mb{P}^{n+1}}\to 0$, the total chern class of $TQ$ is
\begin{align*}
c(TQ)=\frac{(1+H)^{n+2}}{(1+2H)}, 
\end{align*}
where $H$ is the restriction of $\ms{O}_{\mb{P}^{n+1}}(1)$ to $Q$. The obvious problem is that $\frac{1}{1+2H}$ has negative terms when expanded as a power series. Given $0\leq i\leq n$, we will show $c_i(TQ)$ is a positive integer multiple of $H^k$. Consider the difference 
\begin{align*}
\alpha&=[TQ]-[\ms{O}_Q^{n-i+1}]\\
&=[\ms{O}_{Q}(1)^{n+2}]-[\ms{O}_Q(2)]-[\ms{O}_Q^{n-i+2}]
\end{align*}
 in the Grothendieck group of vector bundles over $Q$. Then, by \cite{Aluffitensor},
\begin{align*}
c_i(\alpha)&=c_i(\alpha\otimes \ms{O}_{Q}(-1)),
\end{align*}
and 
\begin{align*}
c(\alpha\otimes \ms{O}_{Q}(-1))&= \frac{(1)^{n+2}}{(1+H)(1-H)^{n-i+2}}=\frac{1}{(1-H^2)(1-H)^{n-i+1}},
\end{align*}
which has positive coefficients.
\end{Exam}

\section{Application 1: immersed curves in generalized flag varieties}
\label{GP}
In this section, we will apply Theorem \ref{posTX} when $X$ is a product of generalized flag varieties $G_i/P_i$. In addition to the conventions assumed in Section \ref{setup}, we will further assume that our base field is $\mb{C}$. 

\subsection{Positivity of the tangent bundle}
Crucially, we will need the positivity of the tangent bundle of a generalized flag variety $G/P$. Recall that the Bruhat decomposition of $G/P$ decomposes $G/P$ into Bruhat cells, each isomorphic to affine space. We let
\begin{enumerate}
\item
$G$ is a complex simple Lie group and $P\subset G$ is a parabolic
\item
$T\subset B\subset P\subset G$, where $T$ is a maximal torus, $B$ is a Borel subgroup, and $P$ is a parabolic
\item
$W=N_G(T)/T$ is the Weyl group, $W_P$ be the subgroup of $W$ generated by the reflections in $P$
\item
$W^P\subset W$ be the subset containing the unique element of minimal length in each right $W_P$-coset in $W$
\item
$e_{w P}=w P/P$ denote the $T$-fixed points of $G/P$, where $w\in W^P$
\item
$C_{w P}=Be_{w P}$ denote the Bruhat cell
\item
$X_{w P}=\ol{C_{w P}}$ denote the Schubert variety
\end{enumerate}

The Schubert varieties, indexed by $W^P$, form a free basis for the Chow group of $G/P$ and the opposite Schubert varieties form a dual basis under the intersection pairing \cite[Lemma 1 (1)]{Brion2}.

\begin{Prop}
\label{posGP}
If we write the total chern class $c(T(G/P))$ of a generalized flag variety as a sum of classes of Schubert cells, every Schubert cell appears with a positive coefficient.
\end{Prop}

\begin{proof}
This is a corollary of \cite[Corollary 1.4]{Aluffi}, which states that the Chern-Schartz-MacPherson (CSM) class of a Bruhat cell in $G/P$ when written in the form
\begin{align*}
c_{SM}(C_{w P})=\sum_{v\in W^P, v\leq w}{c(v;w)[X^{w P}]}
\end{align*}
has the property that $c(v;w)$ is always nonnegative. Since the leading term of $c_{SM}(C_{w P})$ is $[X_{w P}]$, $c(w,w)=1$. Since the CSM classes of the Bruhat cells add to the CSM class of $G/P$, and the CSM class of $G/P$ is the usual chern class of $G/P$, we know that if we write $c(T(G/P))$ in terms of the free basis of Schubert varieties, every schubert variety $[X_{wP}]$ will appear for all $w\in W^P$. 
\end{proof}

\begin{Prop}
\label{posGP2}
Let $X=G_1/P_1\times\cdots\times G_i/P_i$ be a product of generalized flag varieties over $\mb{C}$, then for all subvarieties $Y\subset X$,
\begin{align*}
\int_{G/P}{[Y]\cap c(TX)}>0,
\end{align*}
where $c(TX)=1+c_1(TX)+\cdots$ is the total chern class.
\end{Prop}

\begin{proof}
Since each $G_i/P_i$ has a cellular decomposition by Bruhat cells, so does $X$. The closures of products of Bruhat cells form a free basis for $A_{*}(X)$, and the closures of products of opposite Bruhat cells form a dual basis. Therefore, $[Y]$ is rational equivalent to a nonnegative combination of products of opposite Schubert cells. Since $X$ is projective, $[Y]$ cannot be zero, so if we write $[Y]$ as a linear combination of classes of products of opposite Schubert cells, at least one of the coefficients is positive. 

Finally, Proposition \ref{posGP} shows that if we write 
\begin{align*}
c(TX)=\pi_1^{*}c(T(G_1/P_1))\cdots \pi_{i}^{*}c(T(G_i/P_i))
\end{align*}
as a linear combination of products of Schubert cells, where the $\pi_i$ are the projections $X\to G_i/P_i$, then all the coefficients are positive. Therefore, $c(TX)$ integrates positively when capped with any product of opposite Schubert cells, so 
\begin{align*}
\int_{G_1/P_1\times\cdots\times G_i/P_i}{[Y]\cap c(TX)}>0.
\end{align*}
\end{proof}

Proposition \ref{posGP} and Theorem \ref{posTX} yield
\begin{Thm}
\label{thmGP}
Let $X=G_1/P_1\times\cdots\times G_i/P_i$ be a product of generalized flag varieties over $\mb{C}$. If $X$ is $n$-dimensional, then any complete family of positive genus curves is dimension at most $n-2$. 
\end{Thm}

\begin{proof}
Proposition \ref{posGP2} shows that setting $X=G_1/P_1\times\cdots\times G_i/P_i$ satisfies the conditions of Theorem \ref{posTX}. 
\end{proof}

\section{Application 2: Immersed rational curves in projective spaces}
\label{Pn}
We will now specialize the setup in Section \ref{setup} to the case of rational curves in $X=\mb{P}^n$. We cannot apply Theorem \ref{posTX} in this case because Theorem \ref{posTX} uses the condition of positive genus to conclude $K$ is nef. However, in the case $X=\mb{P}^n$, we can extend the result to genus zero by observing that, except in the case of lines, $K+H$ is nef, where $H$ is the pullback of the hyperplane class in $\mb{P}^n$ to $\ms{C}$. 


We can and will assume $\ms{C}\cong \mb{P}(V)$ for some rank 2 vector bundle $V$ over $B$ \cite[Proposition 2.1]{deLand}.  For the rest of Section \ref{Pn} in addition to the conventions in Section \ref{setup}, we will further assume:
\begin{enumerate}
\item
$X=\mb{P}^n$
\item
$\ms{C}\xrightarrow{\pi} B$ is isomorphic to $\mb{P}(V)\to B$ for a rank 2 vector bundle $V$ over $B$.
\end{enumerate}
In particular, the results of this section do not depend on the characteristic of the base field. Lemma \ref{obstructionPn} can be deduced from Lemma \ref{obstruction} or recalled from \cite[Proof of Theorem]{CR2}.

\begin{Lem} 
\label{obstructionPn}
Given a complete family of curves $f:\ms{C}\to \mb{P}^n$, we must have the identity
\begin{align*}
H^n+H^{n-1}(K+H)+\cdots+(K+H)^n&=0
\end{align*}
on the total space $\ms{C}$, where $H$ is the hyperplane class on $\mb{P}^n$ (pulled back to $\ms{C}$ via $f$). 
\end{Lem}

One can view Lemma \ref{obstructionPn} as saying $T\mb{P}^n$ is positive beyond what Theorem \ref{posTX} requires, which will allow us to extend the argument to genus zero. 

\subsection{$K+H$ is nef}
\begin{Lem}
\label{nef}
Let $B$ be a proper variety, $V$ be a rank 2 vector bundle on $B$, $H$ be a nef divisor on $\mb{P}(V)$, and $K$ be the divisor corresponding to the relative sheaf of differentials of $\mb{P}(V)\xrightarrow{\pi}B$. If $H$ restricts to a divisor of degree at least 2 on each fiber of $\pi$, then $H+K$ is nef. 
\end{Lem}

\begin{proof}
We will show this by intersecting with each test curve. Let $D$ be a smooth curve and $i: D\to \mb{P}(V)$ be a map. 
After pulling back the bundle $\mb{P}(V)\to B$ via the map $D\to B$, we get a ruled surface with a section mapping to $i(D)$ in $\mb{P}(V)=\ms{C}$. To summarize, we have the following diagram
\begin{center}
\begin{tikzcd}
\mb{P}(j^{*}V) \arrow[r, "h"] \arrow[d,"\pi_D"] \arrow[rr, bend left=25, "g"] &  \mb{P}(V) \arrow[d, "\pi"] \arrow[r,"f"] & \mb{P}^n\\
D \arrow[u, bend left=20, "s"] \arrow[r, "j"] \arrow[ru,"i"] & B
\end{tikzcd}
\end{center}
We want to show $\deg(i^{*}(K+H))\geq 0$. In this way, we have reduced our problem to a question about a section on a ruled surface over $D$. 

So now assume $D\cong B$ and $i: D\to \mb{P}(V)$ is a section of $\pi: \mb{P}(V)\to B$. We want to show $(K+H)\cdot D\geq 0$. Since $D$ is a section of $\pi$, the normal bundle of $D$ in $\mb{P}(V)$ is the relative tangent bundle of $\pi$ restricted to $D$, so its first chern class is $-K|_D$. Therefore, 
\begin{align*}
D^2=-K\cdot D\Rightarrow (K+H)\cdot D=-D^2+H\cdot D.
\end{align*}
If $D^2<0$, then $-D^2+H\cdot D\geq -D^2> 0$, so we can assume $D^2\geq 0$. 

By the projective bundle theorem \cite[Theorem 3.3]{Fulton}, the N\'eron-Severi group $N^1(\mb{P}(V))_{\mb{R}}$ is generated by $D$ and a fiber $F$ of $\pi$. Let $H=a(D+bF)$ in $N^1(\mb{P}(V))_{\mb{R}}$, so $b$ is the ``slope'' of $H$ in $N^1(\mb{P}(V))_{\mb{R}}$. By Kleiman's theorem \cite[Theorem 1.4.9]{Pos1}, $H^2\geq 0$, so $b\geq -\frac{D^2}{2}$. 

Now, we expand 
\begin{align*}
(K+H)\cdot D&=(-D+H)\cdot D=(a-1)D^2+ab\\
&\geq (a-1)D^2-a\frac{D^2}{2}=(\frac{a}{2}-1)D^2\geq 0,
\end{align*}
where we used the assumption $a\geq 2$ at the very last step. 
\end{proof}

\subsection{Conclusion of the argument}

The proof of Theorem \ref{posTX} is a simple modification of \cite[Theorem 1]{CR1}, but we include it for the sake of completeness and because the proof is short. 
\begin{Thm}
\label{PnT}
Given a complete family of rational curves in $\mb{P}^n$
\begin{center}
\begin{tikzcd}
\ms{C}\cong \mb{P}(V) \arrow[d,"\pi"] \arrow[r,"f"] & \mb{P}^n\\
B &
\end{tikzcd}
\end{center}
we must have $\dim(B)\leq n-2$, except in the case where $f$ restricts to a degree one map on each fiber.
\end{Thm}

\begin{proof}
Suppose for the sake of contradiction that $\dim(B)\geq n-1$. If necessary, we can slice $B$ by a general hyperplane class until $\dim(B)=n-1$ and $\dim(\ms{C})=n$. We apply Lemma \ref{obstructionPn} to get
\begin{align*}
\int_{\ms{C}}{H^n+H^{n-1}(K+H)+\cdots+(K+H)^n}&=0
\end{align*}
on $\ms{C}$. By Lemma \ref{nef} $K+H$ is nef, so the fact that intersection multiplicities of nef divisors is nonnegative \cite[Example 1.4.16]{Pos1} implies all the terms are nonnegative. We need to show one of the terms is positive. Let $s=\dim(f(\ms{C}))$. Then, if $S=\{p_1,p_2,\ldots, p_{\deg(f(\ms{C}))}\}$ is the finite set of points given by intersecting $f(\ms{C})$ with $s$ general hyperplanes, then
\begin{align*}
\int_{\ms{C}}{H^s(H+K)^{n-s}}=\sum_{p\in S}{\int_{f^{-1}(p)}K^{n-s}},
\end{align*}
and each term in the sum is positive by Corollary \ref{gauss}.
\end{proof}

\appendix
\section{Gauss map and curvilinear schemes}
\label{appendix1}
\begin{Def}
Let $\Curv^m(X)$ be the locus of curvilinear schemes in the Hilbert scheme of length $m$ subschemes of $X$.
\end{Def}

The key fact we will use about curvilinear schemes is the following:
\begin{Lem}
[{\cite[Lemma 1.1]{CR1}}]
\label{ab}
The fibers of the maps $\Curv^{m+1}(X)\to \Curv^m(X)$ are affine spaces for $m\geq 1$. 
\end{Lem}

\begin{Lem}
\label{immersion}
If ${\rm Spec}(k[t]/(t^{m+1}))\to X$ is an injection on tangent spaces, then it is a closed immersion.
\end{Lem}

\begin{proof}
Taking an affine chart around the image of ${\rm Spec}(k[t]/(t^{m+1}))\to X$ reduces us gives us a map of rings $A\to k[t]/(t^{m+1})$. By assumption, $A\to k[t]/(t^2)$ is surjective, so there is an element of the form $t+a_2t^2+\cdots+a_mt^m$ in the image of $A\to k[t]/(t^{m+1})$. Taking powers, we see that $A\to k[t]/(t^{m+1})$ is surjective. 
\end{proof}

\begin{Prop}
\label{fm}
There exist maps $f_m: \ms{C}\to \Curv^m(X)$ lifting $f$ that makes the diagram below compatible:
\begin{center}
\begin{tikzcd}
& \vdots \arrow[d]\\
& \Curv^m(X)\arrow[d]\\
& \Curv^{m-1}(X)\arrow[d]\\
 & \vdots \arrow[d]\\
\ms{C} \arrow[ruuu,"f_m"] \arrow[ruu,swap,"f_{m-1}"] \arrow[r,"f_1=f"]  & \Curv^1(X)=X
\end{tikzcd}
\end{center}
\end{Prop}

\begin{proof}
The compatibility will follow from the construction. Since $\pi: \ms{C}\to B$ is a family of smooth curves, the curvilinear locus in the relative Hilbert scheme parameterizing length $m$ subschemes in the fibers of $\pi$ is isomorphic to $\ms{C}$. Therefore, the universal subscheme $\mc{Z}_m$ of curvilinear schemes can be regarded as a family over $\ms{C}$ and a subscheme of $\ms{C}\times_B\ms{C}$. 
\begin{center}
\begin{tikzcd}
\mc{Z}_m\arrow[d] \arrow[r] & \ms{C} \arrow[r,"f"] \arrow[d,"\pi"]& X\\
\ms{C}\arrow[r,"\pi"] & B &
\end{tikzcd}
\end{center}

Taking the product of the projection map $\mc{Z}_m\to \ms{C}$ from the universal family and the composite of $\mc{Z}_m\to \ms{C}\xrightarrow{f} X$, we get a map $\phi: \mc{Z}_m\to X\times \ms{C}$ as schemes over $\ms{C}$. Let $b\in B$ and $p\in C_b:=\pi^{-1}(b)$. The map $\phi$ restricted to a fiber over $p\in \ms{C}$ is the map from the $m^{th}$ order neighborhood of $p$ in $C_b$ to $X$ under $f$, which is a closed immersion by Lemma \ref{immersion}. We claim that $\phi$ is a closed immersion, so $\mc{Z}_m\subset X\times \ms{C}$ induces a map $f_m: \ms{C}\to \Curv^m(X)$. 

To see $\phi$ is a closed immersion, we first note that topologically it is a homeomorphism onto a graph in $X\times \ms{C}$. Now, $\phi_{*}\ms{O}_{\mc{Z}_m}$ is coherent by properness, so the cokernel $\mc{K}$ of $\ms{O}_{X\times \ms{C}}\to \phi_{*}\ms{O}_{\mc{Z}_m}$ is coherent. We want to show the cokernel is zero. 

The support of $\mc{K}$ is on the graph of $\phi$, which is homeomorphic to $\ms{C}$ under the projection ${\rm pr}: X\times \ms{C}\to \ms{C}$. Therefore, it suffices to show ${\rm pr}_{*}\mc{K}=0$. To show this, we use the fact that $\phi$ is a closed immersion when restricted to each closed point of $\ms{C}$. This means the coherent sheaf ${\rm pr}_{*}\mc{K}$ is zero when restricted to each closed point of $\ms{C}$, so it is zero. 
\end{proof}

\begin{Prop}
\label{HilbIm}
There exists a dense open set $U\subset B$ for which the map $B\dashrightarrow \Hilb_X$ is defined on $U$ and given set-theoretically by sending $b\in U$ to the reduced image $f(\pi^{-1}(b))\subset X$. 
\end{Prop}

\begin{proof}
The map $B\dashrightarrow \Hilb_X$ is defined as follows. The product of the map $f:\ms{C}\to X$ and $\ms{C}\to B$ gives a map $g: \ms{C}\to X\times B$ and a closed subscheme $g(\ms{C})\subset X\times B$. By generic flatness \cite[Tag 0529]{stacks-project}, $g(\ms{C})\to B$ is flat over a dense open $U\subset B$, which defines a map $U\to \Hilb_X$. 

Also, the generic fiber of $g(\ms{C})\to B$ is integral as $g(\ms{C})$ is the closure of the image of the generic point of $\ms{C}$ in $X\times B$. If the characteristic of the base field is zero, then the generic fiber is geometrically integral, so the general fiber of $g(\ms{C})\to B$ is integral \cite[IV 12.1.1 (x)]{EGA} and we can restrict $U$ so that every fiber of $g(\ms{C})|_{U}\to U$ is integral. 

In arbitrary characteristic, $g(\ms{C})$ agrees with the scheme-theoretic image as $\ms{C}$ is reduced \cite[Tag 056B]{stacks-project}. Scheme-theoretic image of a quasicompact morphism commutes with flat base change \cite[Tag 081I]{stacks-project}. We can base change by the map $X_{K}\to X\times B$, where $\Spec(K)\to B$ is a geometric point with image the generic point of $B$. Then, we find the scheme theoretic image $g(\ms{C})$ pulled back to $X_K$ is the scheme theoretic image of an integral scheme over $K$, namely $\ms{C}$ restricted to $\Spec(K)\to B$, so the generic fiber of $g(\ms{C})$ is geometrically integral. Then, we can restrict $U$ so that every fiber of $g(\ms{C})|_{U}\to U$ is integral. 

Finally, over each $b\in U$, $g(\ms{C})$ restricted to $b$ is the reduced image of $\pi^{-1}(b)$ under $f$. 
\end{proof}


The statement of Proposition \ref{TTQQ} is a bit long, but the intent is to say that, to specify an integral curve, it suffices to specify a large curvilinear scheme contained in the integral curve. 
\begin{Prop}
\label{TTQQ}
Suppose we have an open $U\subset B$ as in Proposition \ref{HilbIm}, giving a map $\phi: U\to \Hilb_X$ where $b\in U$ is sent to the reduced image $f(\pi^{-1}(b))\subset X$ and suppose $p\in X$. We can choose $m$ large enough so that the composite $f^{-1}(p)\cap \ms{C}|_{U}\to U\xrightarrow{\phi} \Hilb_X$ factors through $f^{-1}(p)\cap \ms{C}|_{U}\to \ms{C}|_{U}\xrightarrow{f_m} \Curv^m(X)$ set-theoretically.
\begin{center}
\begin{tikzcd}
\ms{C}|_{U}\cap f^{-1}(p) \arrow[r, hook] & \ms{C}|_{U} \arrow[d, "\pi"] \arrow[r,"f_m"]& \Curv^m(X) \\
& U \arrow[r,"\phi"] & \Hilb_X
\end{tikzcd}
\end{center}
Written out, this is the statement that if $f_m(q)=f_m(q')$ for all $q$ and $q'$ in $\ms{C}|_{U}\cap f^{-1}(p)$, then $\phi(\pi(q))=\phi(\pi(q'))$. Here the maps $f_m: \ms{C}\to \Curv^{m}(X)$ are from Proposition \ref{fm}.
\end{Prop}

\begin{proof}
Let $\RHilb_X\subset \Hilb_X$ denote an open locus parameterizing integral curves that are subschemes of $X$, such that $\RHilb_X$ is finite type and contains the image of $\phi$. Let $\Delta_{\RHilb_X}\subset \RHilb_X\times \RHilb_X$ denote the diagonal, ${\rm pr}_1, {\rm pr}_2$ denote the two projections $(\RHilb_X\times\RHilb_X)\backslash \Delta_{\RHilb_X}\to \RHilb_X$ and $\ms{C}_{{\rm univ}}\to \RHilb_X$ denote the universal curve restricted to $\RHilb_X$. 

Since two different integral curves can have intersection at most zero-dimensional, the fibers of ${\rm pr}_1^{*}\ms{C}_{{\rm univ}}\cap {\rm pr}_2^{*}\ms{C}_{{\rm univ}}\to (\RHilb_X\times\RHilb_X)\backslash \Delta_{\RHilb_X}$ has finite fibers. Pushing forward the structure sheaf of ${\rm pr}_1^{*}\ms{C}_{{\rm univ}}\cap {\rm pr}_2^{*}\ms{C}_{{\rm univ}}$ to $(\RHilb_X\times\RHilb_X)\backslash \Delta_{\RHilb_X}$ gives a coherent sheaf $\ms{F}$ over $(\RHilb_X\times\RHilb_X)\backslash \Delta_{\RHilb_X}$. Applying upper semicontinuity of the ranks of the fibers of $\ms{F}$ over the points of $(\RHilb_X\times\RHilb_X)\backslash \Delta_{\RHilb_X}$ together with Noetherian induction shows that the rank of $\ms{F}$ is strictly bounded above by some integer $m$. We claim that this $m$ suffices.

We will show the contrapositive. Suppose we have $q,q'\in \ms{C}|_{U}\cap f^{-1}(p)$ and $\phi(\pi(q))\neq \phi(\pi(q'))$. This means $\phi(\pi(q))$ and $\phi(\pi(q'))$ are integral curves that are subschemes of $X$ that intersect in a finite scheme of length strictly less than $m$. Then, $f_m(q)\neq f_m(q')$. 
\end{proof}


Finally, we prove Lemma \ref{curvilinear}.
\begin{proof}[Proof of Lemma \ref{curvilinear}]
Pick an integer $m$ large enough satisfying the conditions of Proposition \ref{TTQQ}. Consider the maps $f_m$ given in Proposition \ref{fm}. 
\begin{center}
\begin{tikzcd}
& \vdots \arrow[d]\\
& \Curv^m(X)\arrow[d]\\
& \Curv^{m-1}(X)\arrow[d]\\
 & \vdots \arrow[d]\\
\ms{C} \arrow[ruuu,"f_m"] \arrow[ruu,swap,"f_{m-1}"] \arrow[r,"f_2"]  & \Curv^2(X)=\mb{P}(TX)
\end{tikzcd}
\end{center}
Let $\Curv^m_p(X)$ denote the fiber of $\Curv^m_p(X)\to X$ over $p$. Restricting to $f^{-1}(p)$ yields
\begin{center}
\begin{tikzcd}
& \vdots \arrow[d]\\
& \Curv^m_p(X)\arrow[d]\\
& \Curv^{m-1}_p(X)\arrow[d]\\
 & \vdots \arrow[d]\\
f^{-1}(p) \arrow[ruuu,"f_m"] \arrow[ruu,swap,"f_{m-1}"] \arrow[r,"f_2=\gamma"]  & \Curv_p^2(X)=\mb{P}(T_pX)
\end{tikzcd}
\end{center}
Let $U\subset B$ be as in Proposition \ref{HilbIm}. Since $p\in f(\ms{C})$ is general, $U\cap f^{-1}(p)\subset f^{-1}(p)$ is dense. Also the fibers of the map $U\cap f^{-1}(p)\to \Hilb_X$ is generically finite by our assumption in Definition \ref{familydef}. Then, Proposition \ref{TTQQ} implies the fibers of $f_m$ restricted to $U\cap f^{-1}(p)$ are also generically finite. Applying Lemma \ref{ab} shows the map $f_2$ restricted to $U\cap f^{-1}(p)$ is also generically finite, so $\gamma$ is generically finite. 
\end{proof}

\section{Positivity of the obstruction class}
\label{appendix2}
\begin{proof}[Proof of Theorem \ref{posTX}]
Let $\pi: \ms{C}\to B$ and $f: \ms{C}\to G/P$ be our family in terms of Definition \ref{familydef}. Suppose for the sake of contradiction that $\dim(B)\geq n-1$. If necessary, we can slice $B$ by a general hyperplane class until $\dim(B)=n-1$ and $\dim(\ms{C})=n$. We apply Lemma \ref{obstruction} to get
\begin{align*}
(K^n + f^{*}c_1(TX)K^{n-1}+\cdots+f^{*}c_n(TX))\cap [\ms{C}]&=0
\end{align*}
on $\ms{C}$. We will show that in fact 
\begin{align}
\int_{\ms{C}}{(K^n + f^{*}c_1(TX)K^{n-1}+\cdots+f^{*}c_n(TX))\cap [\ms{C}]}&>0\label{positive}
\end{align}

First, we claim 
\begin{align}
\int_{\ms{C}}{f^{*}c_i(TX)\cap K^{n-i}\cap [\ms{C}]}\geq 0 \label{nonnegative}
\end{align}
for each $i$. 

To see this, it suffices to note that $c_i(f^{*}TX)\cap [\ms{C}]$ is effective since $f^{*}TX$ is globally generated \cite[Example 14.4.3 (i)]{Fulton} and then use the fact that $K$ is nef to see that a power of $K$ intersected with an effective cycle of the appropriate dimension is nonnegative \cite[Example 1.4.16]{Pos1}. 

Now, let $b=\dim(f(\ms{C}))$. We claim
\begin{align*}
\int_{\ms{C}}{f^{*}c_b(TX)\cap K^{n-b}\cap [\ms{C}]}> 0.
\end{align*}
Since pushing forward from $\ms{C}$ to a point can be factored as the composite of pushing forward via $f$ and then pushing forward from $X$ to a point, we can apply the projection formula \cite[Example 8.1.7]{Fulton}, to obtain 
\begin{align*}
\int_{\ms{C}}{f^{*}c_b(TX)\cap K^{n-b}\cap [\ms{C}]}=\int_{X}{f_{*}([\ms{C}]\cap K^{n-b})\cap c_b(TX)}. 
\end{align*}
Now, we claim $f_{*}([\ms{C}]\cap K^{n-b})=a[f(\ms{C})]$ for $a>0$ as cycles on $X$. It's clear that $f_{*}([\ms{C}]\cap K^{n-b})=a[f(\ms{C})]$ for some integer $a$ because $f_{*}([\ms{C}]\cap K^{n-b})$ must be supported on $f(\ms{C})$ and has the same dimension, so it remains to show $a$ is positive. To see this, let $H$ be a very ample divisor class on $X$. Then,
\begin{align*}
a &= \frac{\int_{X}{f_{*}([\ms{C}]\cap K^{n-b})\cap H^b}}{\int_{X}{[f(\ms{C})]\cap H^b}}.
\end{align*}
If we choose $b$ general sections of $\ms{O}_X(H)$, $s_1,\ldots,s_b$, 
\begin{align*}
\int_{X}{[f(\ms{C})]\cap H^b}=\#S,
\end{align*}
where $S=f(\ms{C})\cap \{s_1=\cdots=s_b=0\}$. Also, applying the projection formula again yields
\begin{align*}
\int_{X}{f_{*}([\ms{C}]\cap K^{n-b})\cap H^b}&=\int_{\ms{C}}{[\ms{C}]\cap K^{n-b}\cap f^{*}H^b}\\
&=\sum_{p\in S}{\int_{f^{-1}(p)}{K^{n-b}}},
\end{align*}
and each integral $\int_{f^{-1}(p)}{K^{n-b}}$ is positive by Corollary \ref{gauss}. Finally, applying our assumption on $c_b(TX)$, we find 
\begin{align*}
\int_{X}{f_{*}([\ms{C}]\cap K^{n-b})\cap c_b(TX)}=a\int_{X}{[f(\ms{C})]\cap c_b(TX)}>0.
\end{align*}
Finally, adding up the integral \eqref{nonnegative} for each $i$ and noting that the contribution in the case $i=b$ is positive yields \eqref{positive}. 
\end{proof}

\bibliographystyle{plain}
\bibliography{references.bib}

\begin{thebibliography}{10}

\bibitem{Aluffitensor}
Paolo Aluffi and Carel Faber.
\newblock A remark on the {C}hern class of a tensor product.
\newblock {\em Manuscripta Math.}, 88(1):85--86, 1995.

\bibitem{Aluffi}
Paolo Aluffi, Leonardo Mihalcea, Joerg Schuermann, and Changjian Su.
\newblock On degenerate sections of vector bundles.
\newblock {\em preprint}.
\newblock arXiv:1703.10568.

\bibitem{BR}
A.~Borel and R.~Remmert.
\newblock \"uber kompakte homogene {K}\"ahlersche {M}annigfaltigkeiten.
\newblock {\em Math. Ann.}, 145:429--439, 1961/1962.

\bibitem{Brion2}
M.~Brion and V.~Lakshmibai.
\newblock A geometric approach to standard monomial theory.
\newblock {\em Represent. Theory}, 7:651--680, 2003.

\bibitem{CR2}
M.~C. Chang and Z.~Ran.
\newblock Dimension of families of space curves.
\newblock {\em Compositio Math.}, 90(1):53--57, 1994.

\bibitem{CR1}
Mei-Chu Chang and Ziv Ran.
\newblock Closed families of smooth space curves.
\newblock {\em Duke Math. J.}, 52(3):707--713, 1985.

\bibitem{deLand}
Matthew DeLand.
\newblock Complete families of linearly non-degenerate rational curves.
\newblock {\em Canad. Math. Bull.}, 54(3):430--441, 2011.

\bibitem{EGA}
Jean Dieudonn{\'e} and Alexander Grothendieck.
\newblock \'{E}l\'ements de g\'eom\'etrie alg\'ebrique.
\newblock {\em Inst. Hautes \'Etudes Sci. Publ. Math.}, 4, 8, 11, 17, 20, 24,
  28, 32, 1961--1967.

\bibitem{Fulton}
William Fulton.
\newblock {\em Intersection theory}, volume~2 of {\em Ergebnisse der Mathematik
  und ihrer Grenzgebiete. 3. Folge. A Series of Modern Surveys in Mathematics
  [Results in Mathematics and Related Areas. 3rd Series. A Series of Modern
  Surveys in Mathematics]}.
\newblock Springer-Verlag, Berlin, second edition, 1998.

\bibitem{Pos1}
Robert Lazarsfeld.
\newblock {\em Positivity in algebraic geometry. {I}}, volume~48 of {\em
  Ergebnisse der Mathematik und ihrer Grenzgebiete. 3. Folge. A Series of
  Modern Surveys in Mathematics [Results in Mathematics and Related Areas. 3rd
  Series. A Series of Modern Surveys in Mathematics]}.
\newblock Springer-Verlag, Berlin, 2004.
\newblock Classical setting: line bundles and linear series.

\bibitem{Campana}
Roberto Mu\~noz, Gianluca Occhetta, Luis~E. Sol\'a~Conde, Kiwamu Watanabe, and
  Jaros\l aw~A. Wi\'sniewski.
\newblock A survey on the {C}ampana-{P}eternell conjecture.
\newblock {\em Rend. Istit. Mat. Univ. Trieste}, 47:127--185, 2015.

\bibitem{Mumford}
David Mumford.
\newblock Towards an enumerative geometry of the moduli space of curves.
\newblock In {\em Arithmetic and geometry, {V}ol. {II}}, volume~36 of {\em
  Progr. Math.}, pages 271--328. Birkh\"auser Boston, Boston, MA, 1983.

\bibitem{stacks-project}
The {Stacks Project Authors}.
\newblock \textit{Stacks Project}.
\newblock \url{http://stacks.math.columbia.edu}, 2018.

\end{thebibliography}
\end{document}